\documentclass[11pt]{article}

\usepackage{amsmath}
\usepackage{amssymb}
\usepackage{amsthm}
\usepackage{indentfirst}
\usepackage{enumerate}
\usepackage{cite}
\usepackage{arydshln}

\newtheorem{thm}{Theorem}[section]

\newtheorem{prop}{Proposition}[section]
\newtheorem{cor}{Corollary}[section]
\newtheorem{rem}{Remark}[section]

\numberwithin{equation}{section}

\title{\textbf{COSMOLOGICAL MEANING OF GEOMETRIC CURVATURES}}
\author{Nenad O. Vesi\'c\footnote{Serbian Academy of Sciences and Arts, Mathematical
Institute; Serbian Ministry of Education, Science and Technological
Development, Grant No. 174012}}
\date{}

\makeatletter
\def\maketag@@@#1{\hbox{\m@th\normalfont\normalsize#1}}
\makeatother

\begin{document}

  \maketitle
  \begin{abstract}
    In this paper, we analyzed the physical meaning of
    scalar curvatures for a
    generalized Riemannian space. It is developed the Madsen's formulae
    for pressures and energy-densities with respect to the
    corresponding energy-momentum tensors. After that, the
    energy-momentum tensors, pressures, energy-densities and
    state-parameters are analyzed with respect to different concepts of
    generalized Riemannian spaces. At the end of this paper,
    linearities of the energy-momentum tensor, pressure,
    energy-density and the state-parameter are examined.\\[3pt]


    \noindent \textbf{$2010$ Math. Subj. Classification:}
    53B50, 47B15, 53A45, 53B05
  \end{abstract}

  \section{Introduction}

  The main purpose of this paper is to find a physical meaning of
  scalar curvatures for a generalized Riemannian space \cite{eis01}
   and complex or anti-symmetric metrics as
  well.

  \subsection{Physical motivation for differential geometry: basics of cosmology}

  Many geometric papers start with the motivation from General
  Relativity. In the paper \big(Ivanov, Zlatanovi\'c
  \cite{zlativanov}\big), the physical motivation with respect to
  the Einstein's works \cite{e1,e2,e3} is well explained. Some
  other papers where these Einstein's works are cited as the
  motivations for further researches about the spaces with torsion are
  \cite{zlatvstan, zlatvelstan, vesvelstan,
  vesstan, mincvel, zlathintnajd}
  and many others.

  Einstein involved the concept of a complex metric whose real part
  corresponds to the gravity but the imaginer part suits to
  the electromagnetism. Moreover, the affine connection coefficients
  of spaces in Einstein's works are determined by \emph{the Einstein
  Metricity Condition}.

  The Einstein's Theory of General Relativity is a cosmological model
  which was developed. The Kaluza-Klein cosmological model
  \cite{kaluza, klein}
  is one of commonly used models in
  the theory of cosmology. In this model, unlike in the Einstein's one, the
  electromagnetism is covered by the additional dimension of the
  symmetric (real) part of metrics.

  The question that arises is whether the anti-symmetric parts of metric tensors are
  important for any physical application or they are excessive. We
  will physically and geometrically answer to this question in
  this paper.

  The Kaluza-Klein model will not be studied here. The computational
  methodology applied in the book \cite{blau} but also in the article \cite{shapiro}
  combined with torsion tensors will be used in this study.

  \subsection{Geometrical motivation: Generalized Riemannian space}

  An $N$-dimensional manifold $\mathcal M_N$, equipped with a
  non-symmetric metric tensor $\hat g$ of the type $(0,2)$ whose
  components are $g_{ij}$ is the generalized Riemannian space
  $\mathbb{GR}_N$ \big(in the sense of Eisenhart \cite{eis01}\big).
  S. M. Min\v ci\'c \cite{mincic2, mincic4, mincvel}, M. S. Stankovi\'c
  \cite{zlatvstan, zlatvelstan,vesstan, vesvelstan}, Lj. S.
  Velimirovi\'c \cite{vesvelstan, zlatvelstan, mincvel}, M. Lj. Zlatanovi\'c
  \cite{zlatvstan, zlatvelstan, zlativanov} and many
  others have continued the research about these spaces,
  the mappings between them and their generalizations.

  The  symmetric and anti-symmetric part of the tensor $\hat g$ are
  the tensors $\underline{\hat g}=\dfrac12(\hat g+\hat g{}^T)$ and
  $\underset\vee{\hat g}=\frac12(\hat g-\hat g{}^T)$, respectively.
  Their components are

  \begin{eqnarray}
    g_{\underline{ij}}=\frac12(g_{ij}+g_{ji})&\mbox{and}&
    g_{\underline{ij}}=\frac12(g_{ij}+g_{ji})
  \end{eqnarray}

  For our research, the matrix $\big[g_{\underline{ij}}\big]_{N\times N}$
  should be
  non-singular.
  The metric determinant
  for the space $\mathbb R_4$ is $g=\det[g_{\underline{ij}}]$ and it is
  different of $0$ because of the non-singularity of the matrix $\big[g_{\underline{ij}}\big]$. The components $g^{\underline{ij}}$
  for the contravariant symmetric part of the metric
  tensor $\hat g$ are the corresponding elements of
   the inverse matrix $\big[g_{\underline{ij}}\big]_{N\times N}^{-1}$.

  The components of the affine connection coefficients for the space
  $\mathbb{GR}_N$ are the components of the
   generalized Christoffel symbols of the
  second kind

  \begin{equation}
    \Gamma{}^i_{jk}=\frac12g^{\underline{i\alpha}}
    \big(g_{j\alpha,k}-g_{jk,\alpha}+g_{\alpha k,j}\big),
    \label{eq:genChrist2nd}
  \end{equation}

  \noindent for partial derivative $\partial/\partial x^i$ denoted by comma.

  The components of the symmetric and anti-symmetric parts of the
  generalized Christoffel symbol of the second kind are

  \begin{eqnarray}
    \Gamma^i_{\underline{jk}}=\frac12\big(\Gamma^i_{jk}+\Gamma^i_{kj}\big)&\mbox{and}&
    \Gamma^i_{\underset\vee{jk}}=\frac12\big(\Gamma^i_{jk}-\Gamma^i_{kj}\big).
    \label{eq:GammasimantisimGRN}
  \end{eqnarray}

  After some computing, one gets

  \begin{eqnarray}
    \Gamma^i_{\underline{jk}}=
    \frac12g^{\underline{i\alpha}}
    \big(g_{\underline{j\alpha},k}-g_{\underline{jk},\alpha}
    +g_{\underline{\alpha k},j}\big)&\mbox{and}&
    \Gamma^i_{\underset\vee{jk}}=
    \frac12g^{\underline{i\alpha}}
    \big(g_{\underset\vee{j\alpha},k}-g_{\underset\vee{jk},\alpha}
    +g_{\underset\vee{\alpha k},j}\big).
    \label{eq:gammasimantisimGRN(g)}
  \end{eqnarray}

  The doubled components of the anti-symmetric part $\Gamma^i_{\underset\vee{jk}}$
  are the components of the torsion tensor $\hat{T}$ for
  the space $\mathbb{GR}_N$. The components of the torsion tensor
  are  $T^i_{jk}=2\Gamma^i_{\underset\vee{jk}}$. The components of the covariant torsion tensor
  are $T_{ijk}=g_{\underline{i\alpha}}T^\alpha_{jk}$.

  The manifold $\mathcal M_N$ equipped with the tensor
  $\underline{\hat g}$ is the associated space $\mathbb R_N$ of the
  space $\mathbb{GR}_N$. The components of the
  symmetric part (\ref{eq:GammasimantisimGRN})
  of the generalized Christoffel symbols are the Christoffel symbols of the
  second kind. Hence, they are the affine connection coefficients of the
  associated space $\mathbb R_N$.

  The associated space $\mathbb R_N$ is the Riemannian space \big(in the sense of Eisenhart's definition \cite{eis01}\big).
  N. S. Sinyukov \cite{sinjukov}, Josef Mike\v s with
  his research group \cite{miknovi2019,mik2, mik5} and many other authors have developed the theory of Riemannian
spaces.

  With respect to the affine connection of the associated space $\mathbb R_N$
  and a tensor $\hat a$ of the type $(1,1)$, it
  is defined one kind of covariant derivative \cite{miknovi2019,sinjukov, mik2,
  mik5}

  \begin{equation}
    a^i_{j|k}=a^i_{j,k}+\Gamma^i_{\underline{\alpha k}}a^\alpha_j-
    \Gamma^\alpha_{\underline{jk}}a^i_\alpha.
    \label{eq:covderivativeRN}
  \end{equation}

  Based on this covariant derivative, it is founded one
  identity of the Ricci Type. From this identity, it is obtained one
  curvature tensor $\hat R$ of the space $\mathbb{R}_N$ \big(see \cite{miknovi2019,sinjukov, mik2,
  mik5}\big). The components of this tensor are

  \begin{equation}
    R^i_{jmn}=\Gamma^i_{\underline{jm},n}-\Gamma^i_{\underline{jn},m}
    +\Gamma^\alpha_{\underline{jm}}\Gamma^i_{\underline{\alpha n}}
    -\Gamma^\alpha_{\underline{jn}}\Gamma^i_{\underline{\alpha m}}.
    \label{eq:RRN}
  \end{equation}

  The components of the corresponding tensor of the
   Ricci curvature are $R_{ij}=R^\alpha_{ij\alpha}$. The scalar
   curvature of the associated space $\mathbb R_N$ is
   $R=g^{\underline{\alpha\beta}}R_{\alpha\beta}$.

   A. Einstein studied the spaces whose affine connection
   coefficients are not functions of the metric tensor \cite{e1, e2, e3}.
   With respect to the Einstein Metricity Condition

   \begin{equation}
     g_{ij,k}-\Gamma^\alpha_{ik}g_{\alpha
     j}-\Gamma^\alpha_{kj}g_{i\alpha}=0,
     \label{eq:emc}
   \end{equation}

   \noindent as the
   system of differential equations which generate the affine
   connection coefficients for the affine connection space,
   two kinds of covariant derivatives are defined

   \begin{eqnarray}
     a^i_{j\underset1|k}=a^i_{j,k}+
     \Gamma^i_{\alpha k}a^\alpha_j-
     \Gamma^\alpha_{jk}a^i_\alpha&\mbox{and}&
     a^i_{j\underset2|k}=a^i_{j,k}+
     \Gamma^i_{k\alpha}a^\alpha_j-
     \Gamma^\alpha_{kj}a^i_\alpha.
     \label{eq:covderivative12}
   \end{eqnarray}

    M. Prvanovi\'c \cite{mileva1} obtained the fourth curvature
    tensor for a non-symmetric affine connection space.
   S. M. Min\v ci\'c \cite{mincic4, mincic2} defined four kinds of
   covariant derivatives. These four kinds are the covariant derivatives
   (\ref{eq:covderivative12}) and two novel ones

   \begin{eqnarray}
     a^i_{j\underset3|k}=a^i_{j,k}+
     \Gamma^i_{\alpha k}a^\alpha_j-
     \Gamma^\alpha_{kj}a^i_\alpha&\mbox{and}&
     a^i_{j\underset4|k}=a^i_{j,k}+
     \Gamma^i_{k\alpha}a^\alpha_j-
     \Gamma^\alpha_{jk}a^i_\alpha.
     \label{eq:covderivative34}
   \end{eqnarray}

   With respect to these
   four kinds of covariant derivatives,
   S. M. Min\v ci\'c obtained four curvature tensors,
   eight derived curvature tensors and fifteen curvature
   pseudotensors of the space $\mathbb{GR}_N$.
    The components of curvature tensors for the space $\mathbb{GR}_N$ are
   elements of the family

  \begin{equation}
    \aligned
    K{}^i_{jmn}&=R^i_{jmn}+u
     T{}^i_{{jm}|n}
    +u' T{}^i_{{jn}|m}
    +v T{}^\alpha_{{jm}}
     T{}^i_{{\alpha
    n}}
    +v' T{}^\alpha_{{jn}}
     T{}^i_{{\alpha
    m}}
    +w T{}^\alpha_{{mn}}
     T{}^i_{{\alpha
    j}},
    \endaligned\label{eq:pseudoKGRN}
  \end{equation}

  \noindent for the corresponding coefficients $u$, $u'$, $v$, $v'$,
  $w$. Six of them are linearly independent.

  \pagebreak

  The components of Ricci-curvatures for the space $\mathbb{GR}_N$
  are

  \begin{equation}
    \aligned
    K_{ij}&=R_{ij}+u T{}^\alpha_{{ij}|\alpha}
    -(v'+w) T{}^\alpha_{{i\beta}}
     T{}^\beta_{{j\alpha}}.
    \endaligned\label{eq:pseudoRicciGRN}
  \end{equation}

  \noindent Three
  of these tensors are linearly independent.

The family of scalar curvatures $K=g^{\underline{\gamma\delta}}
  K_{\gamma\delta}$ for the space $\mathbb{GR}_N$ is

  \begin{equation}
    K=R-(v'+w)g^{\underline{\gamma\delta}}
    g^{\underline{\alpha\epsilon}}
    g^{\underline{\beta\zeta}}
     T{}_{\alpha\gamma\beta}
     T{}_{\epsilon\delta\zeta}.
    \label{eq:kappagenrelfinal}
  \end{equation}

  \noindent Two of these curvatures are linearly independent.

  In this paper,
  we will stay focused on the space-time $\mathbb{GR}_4$
  equipped with a non-symmetric metric tensor $\hat g$ whose symmetric part is diagonal.

  \subsection{Motivation}

  In the Theory of General Relativity, the Einstein-Hilbert action is the action that yields the Einstein
  field equations through the principle of least action. Let the
  full action of the theory be

  \begin{equation}
    S=\int{d^4x\sqrt{|g|}\left(R+\mathcal L_{M}\right)},
    \label{eq:ehaction4}
  \end{equation}

  \noindent for the scalar curvature $R$ of the associated space $\mathbb R_4$.
  The term $\mathcal L_{M}$ in the last equation is describing matter
  fields.

  The Ricci tensor $R_{ij}$ and the scalar curvature $R$ for  the space $\mathbb{GR}_4$ satisfy the
  Einstein's equations of motion

  \begin{equation}
  R_{ij}-\frac12Rg_{\underline{ij}}=T_{ij},
  \label{eq:eem}
  \end{equation}

  \noindent where $T_{ij}$ are the
  components of the energy-momentum tensor $\hat T$. The last
  equations are generalized by the cosmological constant $\Lambda$
  as \cite{blau, shapiro}

  \begin{equation}
  R_{ij}-\frac12Rg_{\underline{ij}}+\Lambda g_{\underline{ij}}=T_{ij}.
  \label{eq:eemm}
  \end{equation}

  \begin{rem}
    The equation \emph{(\ref{eq:eem})} is obtained from the
    Einstein-Hilbert action\linebreak $S=\int{d^4x\sqrt{|g|}\big(R+\mathcal
    L_M\big)}$ but the equation \emph{(\ref{eq:eemm})} holds
    from the Einstein-Hilbert action\linebreak $S_\Lambda=\int{d^4x\sqrt{|g|}\big(R
    -2\Lambda+\mathcal
    L_M\big)}$, $S=S_0$.

    In general, the components $T_{ij}$ of the corresponding energy-momentum
    tensor $\hat T$ \big(at the right sides of the equations \emph{(\ref{eq:eem},
    \ref{eq:eemm})}\big) are multiplied by the constant $\kappa$,
    $\kappa=8\pi Gc^{-4}$ for the speed of light $c$ and Newton's
    gravitational constant $G$ but we will chose such coordinates to
    be $\kappa=1$ in further research.
  \end{rem}

  The
  Friedman-Lemaitre-Robertson-Walker (FLRW) and
  the Bianchi Type-I cosmological models are solutions of the
  Einstein's equations of motion.

  \pagebreak

  This article is consisted
  of five sections plus conclusion.
  The purposes of this paper are:

  \begin{enumerate}
    \item To recall and develop the Madsen's formulae \cite{madsen1} for pressure,
    energy-density and state parameter,
    \item To correlate the space-time model caused from the article
    \cite{zlativanov} with the Einstein's equations of motion,
    \item To analyze the linearity of energy-momentum tensors,
    pressures, energy-densities and state-parameters under summings
    of fields.
  \end{enumerate}

  \section{Pressure, density, state parameter and Madsen's formulae}

  Based on the action $I=\int{d^4x\sqrt{|g|}\Big(\dfrac12\big(
  \frac{M_P^2}{8\pi}-\xi\phi^2\big)R+\frac12\partial_\alpha\phi
  \partial^\alpha\phi-V[\phi]\Big)}$, the energy-mo\-men\-tum tensor is \big(see \cite{madsen1}\big)

  \begin{equation}
    T_{ij}=\big(1-\xi\phi^2\big)^{-1}
    \Big[S_{ij}+\xi\big\{g_{\underline{ij}}\square(\phi^2)-
    (\phi^2)_{|i|j}\big\}\Big],
    \label{eq:mset}
  \end{equation}

  \noindent for the constant $\xi$, the time-like scalar field
  $\phi$ which has unit magnitude, the operator $\square$ defined as $\square a^i_j=
  g^{\underline{\alpha\beta}}a^i_{j|\alpha|\beta}$ and the tensor $S_{ij}=
  \phi_{,i}\phi_{,j}-\big(\frac12
  g^{\underline{\alpha\beta}}\phi_{,\alpha}\phi_{,\beta}-V[\phi]\big)
  g_{\underline{ij}}$.
  Madsen also chosen the
  units such that $\hbar=c=1$ and $M_P^2=8\pi$ for the Planck mass
  $M_P$.

  In the second section of the paper \cite{madsen1}, Madsen deals
  with the problem of finding a unique vector $\hat u$ related to the scalar
  field $\phi$, appears in the energy-momentum tensor
  (\ref{eq:mset}). The components for this vector are $u^i$ and they satisfy the
  equation

  \begin{equation}
    u^\alpha u_\alpha=1.
    \label{eq:muu}
  \end{equation}

  The components $u^i$ are

  \begin{equation}
    u^i=\big(\partial^\alpha\phi\partial_\alpha\phi\big)^{-1/2}\partial^i\phi.
    \label{eq:mui}
  \end{equation}

  For the symmetric tensor $\hat h$ of the type $(0,2)$
  whose components are \cite{madsen1}

  \begin{equation}
    h_{ij}=g_{\underline{ij}}-u_iu_j,
    \label{eq:mhij}
  \end{equation}

  \noindent the energy-density $\rho$, the pressure $p$, the
  vector-field $\hat q$ such that $u_\alpha q^\alpha=0$ and the tensor
  $\hat\pi$ of the type $(0,2)$ whose components satisfy the equalities $\pi_{i\alpha}u^\alpha=0$ and
  $\pi^\alpha_\alpha=0$, the
  components of the energy-momentum tensor $\hat T$ of the type $(0,2)$
  are \cite{madsen1}

  \begin{equation}
    T_{ij}=\rho u_iu_j+q_iu_j+q_ju_i-\big(ph_{ij}+\pi_{ij}\big).
    \label{eq:msetfactored}
  \end{equation}

  It holds \big(see \cite{madsen1}\big)
  $p=\frac13\Pi^\alpha_\alpha$ and $\rho=T_{\alpha\beta}u^\alpha
  u^\beta$, for $\Pi_{ij}\equiv ph_{ij}+\pi_{ij}=-T_{\alpha\beta}h^\alpha_ih^\beta_j$. After composing the equation (\ref{eq:mhij}) by $g^{\underline{jk}}$,
  one obtains $h^k_i=\delta^k_i-u_iu^k$. If compose the equation
  (\ref{eq:msetfactored}) by the tensor $g^{\underline{ij}}$, use the symmetry
  $T_{ij}=T_{ji}$, the relation $u^\alpha u_\alpha=1$
   and the previously founded components $h^i_j$,
  we will acquire the following expressions

  \begin{eqnarray}
    \Pi_{ij}=-T_{ij}+T_{i\alpha}u^\alpha u_j+
    T_{j\alpha}u^\alpha u_i-T_{\alpha\beta}u^\alpha u^\beta
    u_iu_j&\mbox{and}&
    \Pi^\alpha_\alpha=-T^\alpha_\alpha+T_{\alpha\beta}u^\alpha
    u^\beta.
    \label{eq:PiTracePi}
  \end{eqnarray}

  Hence, the energy-momentum tensor $\hat T$, the pressure $p$, the energy-density
  $\rho$ and the state parameter $\omega$ satisfy the next
  equalities

  \begin{eqnarray}
    p=-\frac13T{}^\alpha_\alpha+\frac13
    T{}_{\alpha\beta}u^\alpha u^\beta,
    &\rho=T{}_{\alpha\beta}u^\alpha u^\beta,
    &\omega=
    -\frac13T{}^\alpha_\alpha
    \big(T{}_{\beta\gamma}u^\beta
    u^\gamma\big)^{-1}+\frac13.
    \label{eq:omega}
  \end{eqnarray}

  In the comoving reference frame, $u^i=\delta^i_0$, the equalities (\ref{eq:omega})
  reduce to

  \begin{eqnarray}
    p_0=-\frac13T{}^\alpha_\alpha+\frac13
    T{}_{00},&
    \rho_0=T{}_{00},
    &\omega_0=
    -\frac13T{}^\alpha_\alpha
    \big(T{}_{00}\big)^{-1}+\frac13.\label{eq:omegacrf}
  \end{eqnarray}

  After composing the equation (\ref{eq:eemm}) by $u^iu^j$ and $g^{\underline{ij}}$
  but using the equation (\ref{eq:muu}) as well, we get
  \begin{eqnarray}
    T_{\alpha\beta}u^\alpha u^\beta=
    R_{\alpha\beta}u^\alpha u^\beta-\frac12R+\Lambda&\mbox{and}&
    T^\alpha_\alpha=-R+4\Lambda.
    \label{eq:RRLambdauug}
  \end{eqnarray}

  In the case of $i=j=0$, one obtains

  \begin{equation}
    T_{00}\overset{(\ref{eq:eemm})}=R_{00}-\frac12Rg_{\underline{00}}+\Lambda
    g_{\underline{00}}.
    \label{eq:RRLambdai0j0}
  \end{equation}

  After substituting the equations (\ref{eq:RRLambdauug},
  \ref{eq:RRLambdai0j0}) into the expressions
  (\ref{eq:omega}, \ref{eq:omegacrf}), we will find

  \setlength\dashlinedash{0.2pt}
\setlength\dashlinegap{1.5pt} \setlength\arrayrulewidth{0.3pt}

  \begin{table}[h]
  \centering
    \begin{tabular}{|l:l|}
    \hline
      In a reference system $u_i$&In the comoving reference system
      $u_i=\delta_{i0}$\\
      \hdashline
      $p=\frac13R_{\alpha\beta}u^\alpha u^\beta
      +\frac16R-\Lambda$&$p_0=\frac13R_{00}+\frac16R-\Lambda$\\
      $\rho=R_{\alpha\beta}u^\alpha u^\beta-
      \frac12R+\Lambda$&$\rho_0=R_{00}-\frac12R+
      \Lambda$\\
      {$\omega=p\rho^{-1}$}&
      {$\omega_0=p_0\rho_0^{-1}$}\\
      \hline
    \end{tabular}
    \caption{Pressures, energy-densities and the state parameters}
    \label{tab:pressdenstp}
  \end{table}

  \section{Generalized Einstein's equations of motion I}

  Let us consider the Einstein-Hilbert action

  \begin{equation}
    S=\int{d^4x\sqrt{|g|}\big(K-2\Lambda\big)},
    \label{eq:ehactiongeneralN}
  \end{equation}

  \noindent with respect to the Shapiro's cosmological model
  \cite{shapiro}.

  Based on the equation (\ref{eq:kappagenrelfinal}), we transform the
  Einstein-Hilbert action (\ref{eq:ehactiongeneralN}) to

  \begin{equation}
    S=\int{d^4x\sqrt{|g|}\big(R-2\Lambda-(v'+w)g^{\underline{\gamma\delta}}
    T{}^\alpha_{\gamma\beta}T{}^\beta_{\alpha\delta}\big)}.
    \label{eq:ehactiongeneralN'}
  \end{equation}

  After lowering the contravariant indices into the last equation,
  one obtains

  \begin{equation}
  \aligned
    S&=\int{d^4x\sqrt{|g|}\big(R-2\Lambda-(v'+w)g^{\underline{\gamma\delta}}
    g^{\underline{\epsilon\alpha}}g^{\underline{\beta\zeta}}
    T{}_{\epsilon\gamma\beta}T{}_{\alpha\delta\zeta}\big)}.
    \endaligned\label{eq:ehactiongeneralN''}
  \end{equation}

\pagebreak

  If compare the variations of the functional (\ref{eq:ehactiongeneralN''})
  and the Einstein-Hilbert
  action\linebreak
  $S=\int{d^4x\sqrt{|g|}\big(R-2\Lambda+\mathcal L_{M}\big)},
  $ we get

  \begin{equation}
    \mathcal L_{M}=-(v'+w)g^{\underline{\gamma\delta}}
    g^{\underline{\epsilon\alpha}}g^{\underline{\beta\zeta}}
    T{}_{\epsilon\gamma\beta}T{}_{\alpha\delta\zeta}.
    \label{eq:LMN}
  \end{equation}

  The variation of the functional $S_1=S_1[\hat{\underline g}]=\int{d^4x\sqrt{|g|}\big(R
  -2\Lambda\big)}$ is \cite{blau}

  \begin{equation}
    {\delta S_1}=
    \int{d^4x\sqrt{|g|}\big(R_{\alpha\beta}-\frac12Rg_{\underline{\alpha\beta}}+\Lambda
    g_{\underline{\alpha\beta}}\big)\delta g^{\underline{\alpha\beta}}}.
    \label{eq:varderivativeS1}
  \end{equation}

  The variation of the functional
  $S_2=S_2[\hat{\underline g}]=\int{d^4x\sqrt{|g|}g^{\underline{\gamma\delta}}
    g^{\underline{\epsilon\alpha}}g^{\underline{\beta\zeta}}
    T{}_{\epsilon\gamma\beta}
    T{}_{\alpha\delta\zeta}}$ is

    \begin{align}
    &{\delta S_2}=
    \int{d^4x\sqrt{|g|}\Big[3\tau_{\alpha\beta}+2\mathcal W_{\alpha\beta}-
    \dfrac12g^{\underline{\gamma\delta}}
    g^{\underline{\epsilon\eta}}
    g^{\underline{\theta\zeta}}T_{\eta\gamma\theta}T_{\epsilon\delta\zeta}g_{\underline{\alpha\beta}}\Big]
    \delta g^{\underline{\alpha\beta}}},
    \label{eq:varderivativeS2}
    \end{align}

    \noindent for $\tau_{ij}=
    \dfrac{\delta g^{\underline{\gamma\delta}}}{\delta g^{\underline{ij}}}
    g^{\underline{\epsilon\alpha}}g^{\underline{\beta\zeta}}
    T{}_{\epsilon\gamma\beta}
    T{}_{\alpha\delta\zeta}$ and $\mathcal W_{ij}=
    g^{\underline{\gamma\delta}}
    g^{\underline{\epsilon\alpha}}
    g^{\underline{\beta\zeta}}
    T{}_{\epsilon\delta\zeta}\dfrac{\delta
    T_{\alpha\gamma\beta}}{\delta g^{\underline{ij}}}$.

    Based on the variation rule, the variational
    derivatives $\delta T{}_{\alpha.\gamma\beta}/\delta
    g^{\underline{ij}}$ are the components $v_{\alpha\gamma\beta ij}$
    for the tensor $\hat v$ of the type $(0,5)$.

    With respect to the equations (\ref{eq:varderivativeS1},
    \ref{eq:varderivativeS2}), we obtain

    \begin{equation*}
      \aligned
      {\delta S}&=
      \int{d^4x\sqrt{|g|}\Big\{R_{\alpha\beta}\!-\!\frac12Rg_{\underline{\alpha\beta}}\!+\!
      \Lambda g_{\underline{\alpha\beta}}\!-\!
      (v'\!+\!w)\Big[3\tau_{\alpha\beta}\!+\!2\mathcal
      W_{\alpha\beta}\!-\!
      \dfrac12g^{\underline{\gamma\delta}}
      g^{\underline{\epsilon\eta}}
      g^{\underline{\theta\zeta}}
      T_{\eta\gamma\theta}
      T_{\epsilon\delta\zeta}g_{\underline{\alpha\beta}}\Big]\Big\}g^{\underline{\alpha\beta}}}.
      \endaligned
    \end{equation*}

    The right side of the last equation vanishes if and only if

    \begin{equation}
    \aligned
      R_{ij}-\frac12Rg_{\underline{ij}}+
      \Lambda g_{\underline{ij}}&=(v'+w)\big(3\tau_{ij}+2\mathcal
      W_{ij}-\dfrac12g^{\underline{\gamma\delta}}
      g^{\underline{\epsilon\alpha}}g^{\underline{\beta\zeta}}
      T_{\alpha\gamma\beta}T_{\epsilon\delta\zeta}g_{\underline{ij}}
      \big).
      \endaligned\label{eq:equationsofmotionfamily}
    \end{equation}

    The family of Einstein's equations of motion is presented
    by the equation (\ref{eq:equationsofmotionfamily}).

    The next theorem holds.

    \begin{thm}
      With respect to the equations of motion
      \emph{(\ref{eq:equationsofmotionfamily})}, the families of the energy-momentum tensors
      and its traces are

      \begin{equation}
      \aligned
        &T{}_{ij}=(v'+w)\big(3\tau_{ij}+2\mathcal
      W_{ij}-\dfrac12g^{\underline{\gamma\delta}}
      g^{\underline{\epsilon\alpha}}g^{\underline{\beta\zeta}}
      T_{\alpha\gamma\beta}T_{\epsilon\delta\zeta}g_{\underline{ij}}
      \big),\\&
      T{}^\alpha_\alpha=(v'+w)\big(3\tau^\alpha_\alpha
      +2\mathcal W^\alpha_\alpha-2g^{\underline{\gamma\delta}}
      g^{\underline{\epsilon\alpha}}g^{\underline{\beta\zeta}}
      T_{\alpha\gamma\beta}T_{\epsilon\delta\zeta}\big),
      \endaligned\label{eq:stressenergytensorexp}
      \end{equation}

      \noindent respectively, for the coefficients $v'$, $w$ and the above
      defined tensors $\hat\tau$ and $\hat{\mathcal W}$ in the analyzed cosmological model.

      With respect to the equations \emph{(\ref{eq:omega}, \ref{eq:omegacrf})} and the
       equalities
      \emph{(\ref{eq:stressenergytensorexp})},
      the families of the pressures and energy-densities are
      \begin{align}
        &{Pressure:}\left\{\begin{array}{l}
        p=-\frac13(v'+w)\Big[3\tau^\alpha_\alpha+2\mathcal
        W^\alpha_\alpha-\dfrac32g^{\underline{\gamma\delta}}g^{\underline{\epsilon\alpha}}g^{\underline{\beta\zeta}}
        T_{\alpha\gamma\beta}T_{\epsilon\delta\zeta}-\big(3\tau_{\alpha\beta}+2\mathcal
        W_{\alpha\beta}\big)u^\alpha u^\beta\Big],\\
        p_0=-\frac13(v'+w)\big(3\tau^\alpha_\alpha+2
        \mathcal W_\alpha^\alpha-\dfrac32g^{\underline{\gamma\delta}}g^{\underline{\epsilon\alpha}}g^{\underline{\beta\zeta}}
        T_{\alpha\gamma\beta}T_{\epsilon\delta\zeta}-3\tau_{00}-2\mathcal W_{00}\big),
        \end{array}\right.\label{eq:pressurethm31}\\
        &
        {Energy-density:}\left\{\begin{array}{l}
        \rho=(v'+w)\Big[\big(3\tau_{\alpha\beta}+2\mathcal W_{\alpha\beta}\big)u^\alpha u^\beta
        -\dfrac32g^{\underline{\gamma\delta}}g^{\underline{\epsilon\alpha}}g^{\underline{\beta\zeta}}
        T_{\alpha\gamma\beta}T_{\epsilon\delta\zeta}\Big],\\
        \rho_0=(v'+w)\Big[\big(3\tau_{00}+2\mathcal W_{00}\big)-\dfrac32g^{\underline{\gamma\delta}}g^{\underline{\epsilon\alpha}}g^{\underline{\beta\zeta}}
        T_{\alpha\gamma\beta}T_{\epsilon\delta\zeta}\Big].
        \end{array}\right.\label{eq:densitythm31}
      \end{align}

      The state-parameters $\omega=p\cdot\rho^{-1}$
      and $\omega_0=p_0\cdot\rho_0^{-1}$ do not depend of
      the coefficients $u$, $u'$, $v$, $v'$, $w$
      which generate curvature tensors
      of the generalized
      Riemannian space $\mathbb{GR}_4$.\qed
    \end{thm}

    \begin{cor}
      The next equations hold

      \begin{align}
        &\dfrac13R_{\alpha\beta}u^\alpha u^\beta\!+\!\dfrac16R\!-\!\Lambda=
        -\frac13(v'\!+\!w)\Big[3\tau^\alpha_\alpha\!+\!2\mathcal
        W^\alpha_\alpha\!-\!\dfrac32g^{\underline{\gamma\delta}}g^{\underline{\epsilon\alpha}}g^{\underline{\beta\zeta}}
        T_{\alpha\gamma\beta}T_{\epsilon\delta\zeta}\!-\!\big(3\tau_{\alpha\beta}\!+\!2\mathcal
        W_{\alpha\beta}\big)u^\alpha u^\beta\Big],\label{eq:pEQM}
        \\
        &R_{\alpha\beta}u^\alpha u^\beta-\dfrac12R+\Lambda=
        (v'+w)\Big[\big(3\tau_{\alpha\beta}+2\mathcal W_{\alpha\beta}\big)u^\alpha u^\beta
        -\dfrac32g^{\underline{\gamma\delta}}g^{\underline{\epsilon\alpha}}g^{\underline{\beta\zeta}}
        T_{\alpha\gamma\beta}T_{\epsilon\delta\zeta}\Big],\label{eq:rhoEQM}
      \end{align}

      \noindent in the reference frame $u^i$ and

      \begin{align}
        &\dfrac13R_{00}+\dfrac16R-\Lambda=
        -\frac13(v'+w)\Big[3\tau^\alpha_\alpha+2\mathcal
        W^\alpha_\alpha-\dfrac32g^{\underline{\gamma\delta}}g^{\underline{\epsilon\alpha}}g^{\underline{\beta\zeta}}
        T_{\alpha\gamma\beta}T_{\epsilon\delta\zeta}-3\tau_{00}-2\mathcal
        W_{00}\Big],
        \label{eq:pEQM0}
        \\
        &R_{00}-\dfrac12R+\Lambda=
        (v'+w)\Big[3\tau_{00}+2\mathcal W_{00}
        -\dfrac32g^{\underline{\gamma\delta}}g^{\underline{\epsilon\alpha}}g^{\underline{\beta\zeta}}
        T_{\alpha\gamma\beta}T_{\epsilon\delta\zeta}\Big],\label{eq:rhoEQM0}
      \end{align}

      \noindent in the comoving reference frame.
    \end{cor}
    \begin{proof}
      After equalizing the expressions of the pressures $p$, $p_0$ and the
      energy-densities $\rho$, $\rho_0$ from the Table
      \ref{tab:pressdenstp} and the equations
      (\ref{eq:pressurethm31}, \ref{eq:densitythm31}), we complete the
      proof for this corollary.
    \end{proof}

    The equations (\ref{eq:pEQM}, \ref{eq:pEQM0}) are \emph{the equations of equilibrium
    for metric with respect to
    the pressure $p$} ($p$EQM). The equations (\ref{eq:rhoEQM},
    \ref{eq:rhoEQM0}) are \emph{the equations of equilibrium for metric with respect to
    the energy-density $\rho$} ($\rho$EQM).

    With respect to the equations (\ref{eq:gammasimantisimGRN(g)},
    \ref{eq:LMN}) and the meaning of the term $\mathcal L_M$,
     the torsion-free affine connection spaces (the
     Riemannian spaces $\mathbb R_4$ are the special ones)
     describe spaces without matter. Hence, the
    anti-symmetric part of the metric tensor $\hat g$ and the torsion tensor of
    the space $\mathbb{GR}_4$ as well are correlated to
    a matter.

    \subsection{Non-symmetric metrics and lagrangian}

    In this part of the paper, we will examine what are components
    for the anti-symmetric part $\underset\vee{\hat g}$ of
    the  metric tensor $\hat g$ which correspond to the
    summand $\mathcal L_M$ in the Einstein-Hilbert action
    $\int{d^4x\sqrt|g|(R+\mathcal L_M)}$.

    Let us consider the non-symmetric metric
    $\hat g$ whose components are

    \begin{equation}
      g=\left[\begin{array}{cccc}
        s_0(t)&n_0(t)&n_1(t)&n_2(t)\\
        -n_0(t)&s_1(t)&n_1(t)&n_2(t)\\
        -n_1(t)&-n_1(t)&s_2(t)&n_3(t)\\
        -n_2(t)&-n_2(t)&-n_3(t)&s_3(t)
      \end{array}\right].
      \label{eq:gnonsymmetricexm}
    \end{equation}

    The components of the symmetric and anti-symmetric parts
    for this metric are

    \begin{eqnarray}
      \underline g=\left[\begin{array}{cccc}
        s_0(t)&0&0&0\\
        0&s_1(t)&0&0\\
        0&0&s_2(t)&0\\
        0&0&0&s_3(t)
      \end{array}\right]&\mbox{and}&
      \underset\vee g=\left[\begin{array}{cccc}
        0&n_0(t)&n_1(t)&n_2(t)\\
        -n_0(t)&0&n_3(t)&n_4(t)\\
        -n_1(t)&-n_3(t)&0&n_5(t)\\
        -n_2(t)&-n_4(t)&-n_5(t)&0
      \end{array}\right].
      \label{eq:gsimantisimexm}
    \end{eqnarray}

    \pagebreak

    The components of the corresponding Christoffel symbols of the first kind are

    \begin{equation}
      \begin{array}{ll}
        \Gamma_{0.\underline{00}}=\frac12s_0'(t),&
        \\
        \Gamma_{0.\underline{11}}=-\frac12s_1'(t),&
        \Gamma_{1.\underline{01}}=\Gamma_{1.\underline{10}}=
        \frac12s_1'(t),\\
        \Gamma_{0.\underline{22}}=-\frac12s_2'(t),&
        \Gamma_{2.\underline{02}}=\Gamma_{2.\underline{20}}=
        \frac12s_2'(t),\\
        \Gamma_{0.\underline{33}}=-\frac12s_3'(t),&
        \Gamma_{3.\underline{03}}=\Gamma_{3.\underline{30}}=
        \frac12s_3'(t),
      \end{array}
      \label{eq:christoffelexm}
    \end{equation}

    \noindent but $\Gamma_{i.\underline{jk}}=0$ in all other cases.

    The components of the covariant torsion tensor are

    \begin{equation}
    \aligned
      &T_{012}=-T_{021}=-T_{102}=T_{120}=T_{201}=-T_{210}=-n_3'(t),\\
      &T_{013}=-T_{031}=-T_{103}=T_{130}=T_{301}=-T_{310}=-n_4'(t),\\
      &T_{023}=-T_{032}=-T_{203}=T_{230}=T_{302}=-T_{320}=-n_5'(t),
    \endaligned
    \label{eq:torsionexm}
    \end{equation}

    \noindent and $T_{ijk}=0$ otherwise. As we may see, the components of
    the torsion tensor $\hat T$ do not depend of the components
    $n_0(t)$, $n_1(t)$, $n_2(t)$ of the anti-symmetric part of the
    metric tensor $\hat g$.

    With respect to the equation (\ref{eq:ehactiongeneralN''}), we
    obtain that the term $\mathcal L_M$ is

    \begin{equation}
    \aligned
      \mathcal L_M&=-\frac32(v'+w)\tilde g{}^{-1}s_0(t)\Big\{
      s_3(t)\big(n_3'(t)\big)^2
      +s_2(t)\big(n_4'(t)\big)^2
      +s_1(t)\big(n_5'(t)\big)^2\Big\}.
      \endaligned\label{eq:LMexm}
    \end{equation}

    After basic computing, one gets

    \begin{footnotesize}
    \begin{align}
      &\left\{\begin{array}{cccc}
        \multicolumn{4}{l}{\tau_{00}=6
        \big(s_1(t)\big)^{-1}
        \big(s_2(t)\big)^{-1}
        \big(s_3(t)\big)^{-1}\Big(
        \big(n_3'(t)\big)^2s_3(t)+
        \big(n_4'(t)\big)^2s_2(t)+
        \big(n_5'(t)\big)^2s_1(t)
        \Big),}\\
        \multicolumn{4}{l}{\tau_{11}=6\big(s_0(t)\big)^{-1}
        \big(s_2(t)\big)^{-1}\big(s_3(t)\big)^{-1}
        \Big(\big(n_3'(t)\big)^2s_3(t)+
        \big(n_4'(t)\big)^2s_2(t)
        \Big),}\\
        \multicolumn{4}{l}{\tau_{22}=6\big(s_0(t)\big)^{-1}
        \big(s_1(t)\big)^{-1}
        \big(s_3(t)\big)^{-1}\Big(\big(n_3'(t)\big)^2s_3(t)+
        \big(n_5'(t)\big)^2s_1(t)
        \Big),}\\
        \multicolumn{4}{l}{\tau_{33}=6\big(s_0(t)\big)^{-1}
        \big(s_1(t)\big)^{-1}
        \big(s_2(t)\big)^{-1}\Big(\big(n_4'(t)\big)^2s_2(t)+
        \big(n_5'(t)\big)^2s_1(t)
        \Big),}\\
        \multicolumn{2}{l}{\tau_{12}=\tau_{21}=6\big(s_0(t)\big)^{-1}
        \big(s_3(t)\big)^{-1}n4'(t)n_5'(t),}&
        \multicolumn{2}{l}{\tau_{23}=\tau_{32}
        =6\big(s_0(t)\big)^{-1}\big(s_1(t)\big)^{-1}n_3'(t)n_4'(t),}\\
        \multicolumn{4}{l}{\tau_{13}=\tau_{31}=-6
        \big(s_0(t)\big)^{-1}\big(s_2(t)\big)^{-1}n_3'(t)n_5'(t),}\\
        \multicolumn{4}{l}{\aligned\tau^\alpha_\alpha&=
        \big(s_0(t)\big)^{-1}
        \big(s_1(t)\big)^{-1}
        \big(s_2(t)\big)^{-1}
        \big(s_3(t)\big)^{-1}
        \Big\{\big(n_3'(t)\big)^2s_3(t)\big(s(t)-s_3(t)\big)\\&+
        \big(n_4'(t)\big)^2s_2(t)\big(s(t)-s_2(t)\big)+
        \big(n_5'(t)\big)^2s_1(t)\big(s(t)-s_1(t)\big)\Big\},\endaligned}
      \end{array}\right.\label{eq:tau2}\\&\mathcal W_{ij}=
      \sum_{\alpha=1}^4{\sum_{\beta=1}^4{\sum_{\gamma=1}^4{\sum_{\delta=1}^4{\sum_{\epsilon=1}^4
      {\sum_{\zeta=1}^4{\delta^{(\gamma)}_{(\delta)}\delta^{(\epsilon)}_{(\alpha)}
      \delta^{(\beta)}_{(\zeta)}
      \big(s_\alpha(t)\big)^{-1}\big(s_\beta(t)\big)^{-1}
      \big(s_\gamma(t)\big)^{-1}T_{(\epsilon).(\delta)(\zeta)}v_{(\alpha)(\gamma)(\beta)
      ij}}}}}}},\label{eq:W2}
    \end{align}
    \end{footnotesize}

    \noindent for $s(t)=s_0(t)+s_1(t)+s_2(t)+s_3(t)$,
     the above defined tensor $\hat v$ and $\tau_{ij}=0$ in
    all other cases. The brackets
    about the indices in the equation (\ref{eq:W2}) mean that the
    Einstein's Summation Convention should not be applied to them.

    It holds the next theorem.

    \begin{thm}
      The functions $s_0(t)$, $s_1(t)$, $s_2(t)$, $s_3(t)$, $n_3(t)$, $n_4(t)$,
      $n_5(t)$ are the components of a metric tensor \emph{(\ref{eq:gnonsymmetricexm})} which
      corresponds to the Einstein-Hilbert action $S=\int{d^4x\sqrt{|g|}\big(R+\mathcal L_M\big)}$,
      for $\mathcal L_M$ given by the equation \emph{(\ref{eq:LMexm})} if and only if they
      satisfy all of the equations of motion
      \emph{(\ref{eq:equationsofmotionfamily})}.\qed
    \end{thm}

    In the sense of the research in this subsection, the equations
    of motion (\ref{eq:equationsofmotionfamily}) are differential equations
    by the functions
    $n_3(t)$, $n_4(t)$, $n_5(t)$ with respect to the known functions
    $s_0(t)$, $s_1(t)$, $s_2(t)$, $s_3(t)$. They express the correlation between
    the energy-momentum tensor with respect to the symmetric and
    anti-symmetric parts of metrics.

    Let us consider a functional proportion
    $n_3'(t):n_4'(t):n_5'(t)=\alpha_3:\alpha_4:\alpha_5$.
    Hence, it
    exists a non-trivial function $n(t)$ such that
    $n_k'(t)=\alpha_kn(t)$, $k=3,4,5$. With respect to these changes, the equation (\ref{eq:LMexm})
    transforms to

    \begin{equation}
      \mathcal L_M=-\dfrac32(v'+w)
      g^{-1}s_0(t)\big[\alpha_3^2s_3(t)+\alpha_4^2s_2(t)+\alpha_5^2s_1(t)\big]\big(n(t)\big)^2.
      \tag{\ref{eq:LMexm}'}\label{eq:LMexm'}
    \end{equation}

    If $\mathcal L_M\neq0$, the last equation
    involving $n(t)$ as the unknown has two solutions if and only
    if\linebreak
    $(v'+w)\big[\alpha_3^2s_3(t)+\alpha_4^2s_2(t)+\alpha_5^2s_1(t)\big]\neq0$.
    In the case of $(v'+w)\big[\alpha_3^2s_3(t)+\alpha_4^2s_2(t)+\alpha_5^2s_1(t)\big]\mathcal L_M\neq0$ and with respect to the Existence and Uniqueness Theorem, the
    differential equation (\ref{eq:LMexm}) has two solutions by the
    functions $\big(n_3(t),n_4(t),n_5(t)\big)$. These solutions are

    \begin{footnotesize}
    \begin{eqnarray}
      n_{k_0}(t)=\int{\alpha_k\sqrt{-\dfrac2{3(v'+w)}\big(s_0(t)\big)^{-1}
      \big[\alpha_3^2s_3(t)+\alpha_4^2s_2(t)+\alpha_5^2s_1(t)\big]^{-1} g\mathcal L_M}
      dt},&
      n_{k_1}(t)=-n_{k_0}(t).
      \label{eq:solutionnk}
    \end{eqnarray}
    \end{footnotesize}

    In other words, the
    3-tuples $\big(n_{3_0}(t),n_{4_0}(t),n_{5_0}(t)\big)$ and
    $\big(n_{3_1}(t),n_{4_1}(t),n_{5_1}(t)\big)$ are the corresponding components
    of the anti-symmetric part of the metric tensor $\hat g$.

    After substituting the expressions (\ref{eq:tau2}, \ref{eq:W2})
    into the equations (\ref{eq:pressurethm31},
    \ref{eq:densitythm31}), one gets the corresponding pressures $p$
    and $p_0$ and the densities $\rho$ and $\rho_0$ as well. The
    proportions $p\cdot\rho^{-1}$ and $p_0\cdot\rho_0^{-1}$ are the
    corresponding state parameters.

    \section{Generalized Einstein's equations of motion II}

    I. Shapiro \cite{shapiro} studied the theory of gravity with
    torsion. He analyzed the four-dimensional space-time cosmological models. Into the
    second section of the paper \cite{shapiro}, I. Shapiro
    considered the non-symmetric affine connection spaces whose affine
    connection coefficients are $\tilde\Gamma{}^i_{jk}=
    \Gamma^i_{\underline{jk}}+\mathcal K^i_{jk}$,
    for the Christoffel symbols $\Gamma^i_{\underline{jk}}$
     (eq. \ref{eq:gammasimantisimGRN(g)}, left) and the tensor $\hat{\mathcal K}$
     of the type $(1,2)$ whose components satisfy the equality $\mathcal K^i_{jk}=
    -\mathcal K^i_{kj}$.
    The torsion tensor
    $\hat{\tilde T}$ is $\hat{\tilde
    T}=2\hat{\mathcal K}$ in the Shaprio's article \cite{shapiro}.

    We will analyze a generalization of this concept below.

    Fourteen years after Shapiro, S. Ivanov and M. Lj. Zlatanovi\'c
    published the paper \cite{zlativanov} where they involved the
    model of the generalized Riemannian space that generalizes the
    Eisenhart's one.

    We considered above the Einstein's equations of motion covered by the
    generalized Riemannian space with respect to Eisenhart's
    definition \cite{eis01}. In this section, we will derive the equations of motion
    with respect to the generalized Riemannian space
    $\mathbb{G\tilde{R}}{}_4$ defined by S. Ivanov and
    M. Zlatanovi\'c in \cite{zlativanov} as the manifold $\mathcal M_4$
    equipped with non-symmetric metric tensor $\hat g$.

    The covariant affine connection coefficients
    $\tilde\Gamma{}_{ijk}$ for the space $\mathbb{G\tilde R}{}_4$
    are \cite{zlativanov}

    \begin{equation}
    \aligned
      \tilde\Gamma{}_{ijk}&=\Gamma_{i\underline{jk}}+
      \frac12\Big[\tilde T_{jki}+\tilde T_{ijk}-\tilde T_{kij}\Big]-
      \frac12\big[g_{\underline{ki}\underset1|j}+
      g_{\underline{ij}\underset1|k}-
      g_{\underline{kj}\underset1|i}\big],
    \endaligned\label{eq:gammaGRNIZ}
    \end{equation}

    \noindent for the Christoffel symbol of the first kind
    $\Gamma_{i\underline{jk}}$ obtained with respect to the symmetric metric
    tensor $\underline{\hat g}$ and the covariant derivative
$a^i_{j\underset1|k}=a^i_{j,k}+
      \tilde\Gamma{}^i_{{\alpha k}}a^\alpha_j-
      \tilde\Gamma{}^\alpha_{{jk}}a^i_\alpha$.

    With respect to the equation (\ref{eq:gammaGRNIZ}), we obtain

    \begin{eqnarray}
      \tilde\Gamma{}_{i\underline{jk}}=\Gamma_{i\underline{jk}}-
      \frac12\Big[\tilde T_{jik}+\tilde T_{kij}\Big]-
      \frac12\big[g_{\underline{ki}\underset1|j}+
      g_{\underline{ij}\underset1|k}-
      g_{\underline{kj}\underset1|i}\big]&\mbox{and}&
      \tilde\Gamma_{i\underset\vee{jk}}=\frac12\tilde T{}_{ijk}.
      \label{eq:gammasimantisimIZ}
    \end{eqnarray}

    After rising the index $i$ in the last equation, we get

    \begin{eqnarray}
      \tilde\Gamma{}^i_{\underline{jk}}=\Gamma^i_{\underline{jk}}
      -\frac12g^{\underline{i\alpha}}\Big[\tilde T_{j\alpha k}+\tilde T_{k\alpha j}\Big]-
      \frac12g^{\underline{i\alpha}}\big[g_{\underline{k\alpha}\underset1|j}+
      g_{\underline{\alpha j}\underset1|k}-
      g_{\underline{kj}\underset1|\alpha}\big]&\mbox{and}&
      \tilde\Gamma{}^i_{\underset\vee{jk}}=\frac12\tilde T^i_{jk}.
      \label{eq:simgammagrnIZ}
    \end{eqnarray}

    The covariant derivative $\tilde|$ with respect to the
    symmetric affine connection coefficients $\tilde\Gamma{}^i_{\underline{jk}}$
    and the covariant
    derivative (\ref{eq:covderivativeRN}) satisfy the equation

    \begin{equation}
    \aligned
      a^i_{j\tilde|k}=a^i_{j|k}&-\frac12g^{\underline{i\alpha}}
      \Big(\tilde T_{\beta\alpha k}+\tilde T_{k\alpha\beta}+
      g_{\underline{k\alpha}\underset1|\beta}+
      g_{\underline{\alpha\beta}\underset1|k}-
      g_{\underline{k\beta}\underset1|\alpha}\Big)a^\beta_j\\&+
      \frac12g^{\underline{\alpha\beta}}
      \Big(\tilde T_{j\alpha k}+\tilde T_{k\alpha j}+
      g_{\underline{k\alpha}\underset1|j}+
      g_{\underline{\alpha j}\underset1|k}-
      g_{\underline{kj}\underset1|\alpha}\Big)a^i_\beta.
    \endaligned\label{eq:|tilde=|}
    \end{equation}

    The components of the curvature tensor $\hat{\tilde R}$ for the
    associated space $\mathbb{\tilde R}{}_4$ are

    \begin{equation}
      \tilde R{}^i_{jmn}=\tilde\Gamma{}^i_{\underline{jm},n}-
      \tilde\Gamma{}^i_{\underline{jn},m}+
      \tilde\Gamma{}^\alpha_{\underline{jm}}\tilde\Gamma{}^i_{\underline{\alpha
      n}}-
      \tilde\Gamma{}^\alpha_{\underline{jn}}\tilde\Gamma{}^i_{\underline{\alpha
      m}}.
      \label{eq:RRNIZ}
    \end{equation}

    These components and the components
    (\ref{eq:RRN}) of the curvature tensor $\hat R$ for the space
    $\mathbb R_4$ satisfy the equation

    \begin{equation}
      \aligned
      \tilde R{}^i_{jmn}&=R^i_{jmn}-\frac12\eta^i_{jm,n}+
      \frac12\eta^i_{jn,m}-\frac12\big(\eta^\alpha_{jm}\Gamma^i_{\underline{\alpha n}}
      +\eta{}^i_{\alpha n}\Gamma^\alpha_{\underline{jm}}-
      \eta^\alpha_{jn}\Gamma^i_{\underline{\alpha m}}-
      \eta^i_{\alpha m}\Gamma^\alpha_{\underline{jn}}\big)\\&+
      \frac14\big(\eta^\alpha_{jm}\eta^i_{\alpha n}-
      \eta^\alpha_{jn}\eta^i_{\alpha m}\big),
      \endaligned\label{eq:R=RRNIZ}
    \end{equation}

    \noindent for

    \begin{equation}
      \eta^i_{jk}=g^{\underline{i\alpha}}\Big(\tilde T_{j\alpha k}+\tilde T_{k\alpha j}+
      g_{\underline{k\alpha}\underset1|j}+
      g_{\underline{\alpha j}\underset1|k}-
      g_{\underline{kj}\underset1|\alpha}\Big).
      \label{eq:ZIeta}
    \end{equation}

    With respect to the equations (\ref{eq:covderivativeRN},
    \ref{eq:R=RRNIZ}), we proved the next proposition.

    \begin{prop}
      The components of the curvature tensors $\hat{\tilde R}$ and
      $\hat R$ respectively given by the equations
      \emph{(\ref{eq:RRNIZ})} and \emph{(\ref{eq:RRN})} satisfy the
      equation

      \begin{equation}
        \tilde
        R{}^i_{jmn}=R^i_{jmn}-\frac12\eta^i_{jm|n}+\frac12\eta{}^i_{jn|m}+
        \frac14\big(\eta^\alpha_{jm}\eta^i_{\alpha n}-
        \eta^\alpha_{jn}\eta^i_{\alpha m}\big),
        \label{eq:RIZ=Rprop}
      \end{equation}

      \noindent for the components $\eta^i_{jk}$ of the tensor $\hat\eta$ of
      the type $(1,2)$, defined by the
      equation \emph{(\ref{eq:ZIeta})}.\qed
    \end{prop}

    The space $\mathbb{G\tilde R}{}_4$ is a special affine
    connection space. For this reason, the components of the
    curvature tensors $\hat{\tilde K}$ of
    this space are elements of the family \cite{mincic2,mincic4}

    \begin{equation}
      \aligned
      \tilde K{}^i_{jmn}&=\tilde R^i_{jmn}+
      u\tilde T^i_{jm\tilde|n}+u'\tilde T^i_{jn\tilde|m}+
      v\tilde T^\alpha_{jm}\tilde T^i_{\alpha n}+
      v'\tilde T^\alpha_{jn}\tilde T^i_{\alpha m}+
      w\tilde T^\alpha_{mn}\tilde T^i_{\alpha j}.
            \endaligned\label{eq:generalcurvatureZI}
    \end{equation}

    \pagebreak

    After applying the equations (\ref{eq:|tilde=|},
    \ref{eq:RIZ=Rprop}) and the equality
    $\tilde\Gamma{}^i_{\underline{jk}}=\Gamma^i_{\underline{jk}}-\frac12\eta^i_{jk}$
    as well, one proves the next proposition.

    \begin{prop}
      The family of components
      of the curvature tensors $\hat{\tilde K}$ for the generalized
      Riemannian space $\mathbb{G\tilde R}{}_4$ is

      \begin{equation}
        \aligned
        \tilde K{}^i_{jmn}&=R^i_{jmn}-\frac12\eta^i_{jm|n}+
        \frac12\eta{}^i_{jn|m}+\frac14\big(\eta^\alpha_{jm}\eta^i_{\alpha
        n}-\eta^\alpha_{jn}\eta^i_{\alpha m}\big)\\&+
        u\tilde T^i_{jm|n}+u'\tilde T^i_{jn|m}
        +v\tilde T^\alpha_{jm}\tilde T^i_{\alpha n}
        +v'\tilde T^\alpha_{jn}\tilde T^i_{\alpha m}+
        w\tilde T^\alpha_{mn}\tilde T^i_{\alpha j}\\&-\frac u2g^{\underline{i\alpha}}
        \big(\tilde T{}_{\beta\alpha n}+\tilde T{}_{n\alpha\beta}+
        g_{\underline{n\alpha}\underset1|\beta}+
        g_{\underline{\alpha\beta}\underset1|n}-g_{\underline{n\beta}\underset1|\alpha}\big)
        \tilde T{}^\beta_{jm}\\&-
        \frac u2g^{\underline{\alpha\beta}}
        \big(\tilde T{}_{j\alpha n}+
        \tilde T{}_{n\alpha j}+g_{\underline{n\alpha}\underset1|j}
        +g_{\underline{\alpha j}\underset1|n}
        -g_{\underline{nj}\underset1|\alpha}\big)\tilde T^i_{m\beta}
        \\&-\frac{u'}2g^{\underline{i\alpha}}
        \big(\tilde T{}_{\beta\alpha m}+\tilde T{}_{m\alpha\beta}+
        g_{\underline{m\alpha}\underset1|\beta}+
        g_{\underline{\alpha\beta}\underset1|m}-g_{\underline{m\beta}\underset1|\alpha}\big)
        \tilde T{}^\beta_{jn}\\&-
        \frac{u'}2g^{\underline{\alpha\beta}}
        \big(\tilde T{}_{j\alpha m}+
        \tilde T{}_{m\alpha j}+g_{\underline{m\alpha}\underset1|j}
        +g_{\underline{\alpha j}\underset1|m}
        -g_{\underline{mj}\underset1|\alpha}\big)\tilde T^i_{n\beta}
        \\&+
        \frac{u+u'}2g^{\underline{\alpha\beta}}
        \big(\tilde T{}_{m\alpha n}+
        \tilde T{}_{n\alpha m}+g_{\underline{n\alpha}\underset1|m}
        +g_{\underline{\alpha m}\underset1|n}
        -g_{\underline{nm}\underset1|\alpha}\big)\tilde T^i_{j\beta
        }
        .
        \endaligned\label{eq:kappaZI=R+}
      \end{equation}

       The corresponding family of components $\tilde K{}_{ij}=
      \tilde K{}^\alpha_{ij\alpha}$ of the Ricci curvatures is

        \begin{equation}
        \aligned
        \tilde K{}_{ij}&=R_{ij}-\frac12\eta^\alpha_{ij|\alpha}+
        \frac12\eta{}^\alpha_{i\alpha|j}+\frac14\big(\eta^\alpha_{ij}\eta^\beta_{\alpha
        \beta}-\eta^\alpha_{i\beta}\eta^\beta_{j\alpha}\big)\\&+
        u\tilde T^\alpha_{ij|\alpha}+u'\tilde T^\alpha_{i\alpha|j}
        +v\tilde T^\alpha_{ij}\tilde T^\beta_{\alpha\beta}
        -(v'+w)\tilde T^\alpha_{i\beta}\tilde T^\beta_{j\alpha}
        \\&-\frac u2g^{\underline{\alpha\gamma}}
        \big(\tilde T{}_{\gamma\alpha\beta}+
        g_{\underline{\gamma\alpha}\underset1|\beta}+
        g_{\underline{\alpha\beta}\underset1|\gamma}-g_{\underline{\gamma\beta}\underset1|\alpha}\big)
        \tilde T{}^\beta_{ij}\\&-
        \frac u2g^{\underline{\alpha\beta}}
        \big(\tilde T{}_{i\alpha\gamma}+
        \tilde T{}_{\gamma\alpha i}+g_{\underline{\gamma\alpha}\underset1|i}
        +g_{\underline{\alpha i}\underset1|\gamma}
        -g_{\underline{\gamma i}\underset1|\alpha}\big)\tilde T^\gamma_{j\beta
        }
        \\&+
        \frac{u}2g^{\underline{\alpha\beta}}
        \big(\tilde T{}_{j\alpha\gamma}+
        \tilde T{}_{\gamma\alpha j}+g_{\underline{\gamma\alpha}\underset1|j}
        +g_{\underline{\alpha j}\underset1|\gamma}
        -g_{\underline{\gamma j}\underset1|\alpha}\big)\tilde T^\gamma_{i\beta
        }\\&+
        \frac{u'}2g^{\underline{\alpha\beta}}
        \big(\tilde T{}_{i\alpha j}+
        \tilde T{}_{j\alpha i}+g_{\underline{j\alpha}\underset1|i}
        +g_{\underline{\alpha i}\underset1|j}
        -g_{\underline{ji}\underset1|\alpha}\big)\tilde T^\gamma_{\beta
        \gamma}
        .
        \endaligned\label{eq:RiccikappaZI=R+}
      \end{equation}

      The family $\tilde K=g^{\underline{\gamma\delta}}
      \tilde K{}_{\gamma\delta}$ of scalar curvatures of the space
      $\mathbb{G\tilde R}{}_4$ is

      \begin{equation}
        \aligned
        \tilde K&=R-\frac12g^{\underline{\beta\gamma}}\eta^\alpha_{\beta\gamma|\alpha}+
        \frac12g^{\underline{\beta\gamma}}\eta{}^\alpha_{\alpha\beta|\gamma}
        +\frac14g^{\underline{\gamma\delta}}\big(\eta^\alpha_{\gamma\delta}\eta^\beta_{\alpha
        \beta}-\eta^\alpha_{\beta\gamma}\eta^\beta_{\alpha\delta}\big)-(v'+w)g^{\underline{\gamma\delta}}\tilde T^\alpha_{\gamma\beta}\tilde T^\beta_{\delta\alpha}\\&+
        u'g^{\underline{\beta\gamma}}\tilde T^\alpha_{\beta\alpha|\gamma}
                +
        \frac{u'}2g^{\underline{\alpha\beta}}
        g^{\underline{\delta\epsilon}}
        \big(\tilde T{}_{\delta\alpha\epsilon}+
        \tilde T{}_{\epsilon\alpha\delta}+g_{\underline{\epsilon\alpha}\underset1|\delta}
        +g_{\underline{\alpha\delta}\underset1|\epsilon}
        -g_{\underline{\epsilon\delta}\underset1|\alpha}\big)\tilde T^\gamma_{\beta
        \gamma}
        .
        \endaligned\label{eq:scalarcurvatureIZ}
      \end{equation}

In the equations \emph{(\ref{eq:kappaZI=R+},
\ref{eq:RiccikappaZI=R+}, \ref{eq:scalarcurvatureIZ})}, $u, u', v,
v', w$ are the corresponding coefficients. \qed
    \end{prop}

    As we may see, the part of the scalar curvature $\tilde K$ which corresponds to the
    matter is

    \begin{equation}
      \aligned
      \tilde{\mathcal L}{}_{M}&=-\frac12g^{\underline{\beta\gamma}}\eta^\alpha_{\beta\gamma|\alpha}+
        \frac12g^{\underline{\beta\gamma}}\eta{}^\alpha_{\alpha\beta|\gamma}
        +\frac14g^{\underline{\gamma\delta}}\big(\eta^\alpha_{\gamma\delta}\eta^\beta_{\alpha
        \beta}-\eta^\alpha_{\beta\gamma}\eta^\beta_{\alpha\delta}\big)-(v'+w)g^{\underline{\gamma\delta}}\tilde T^\alpha_{\gamma\beta}\tilde T^\beta_{\delta\alpha}
\\&+
        u'g^{\underline{\beta\gamma}}\tilde T^\alpha_{\beta\alpha|\gamma}
        +
        \frac{u'}2g^{\underline{\alpha\beta}}
        g^{\underline{\delta\epsilon}}
        \big(\tilde T{}_{\delta\alpha\epsilon}+
        \tilde T{}_{\epsilon\alpha\delta}+g_{\underline{\epsilon\alpha}\underset1|\delta}
        +g_{\underline{\alpha\delta}\underset1|\epsilon}
        -g_{\underline{\epsilon\delta}\underset1|\alpha}\big)\tilde T^\gamma_{\beta
        \gamma}.
      \endaligned\label{eq:LMIZ}
    \end{equation}

    Let be $\tilde{\mathcal L}_{M}=\tilde {\mathcal L}[u',v',w]$,
    for the coefficients $u'$, $v'$, $w$. The
    meaning of the square brackets in the last equality is that the field
    $\tilde{\mathcal L}$ depends of the linear
    combination of necessary terms with respect to the coefficients $u'$, $v'$, $w$.
    In this case, the full lagrangian with torsion is
$
      \mathcal L=\big(R-2\Lambda+\tilde{\mathcal L}[u',v',w]\big)\sqrt{|g|}$.

    The corresponding Einstein-Hilbert action with torsion is

    \begin{equation}
      \tilde S=\int{d^4x\sqrt{|g|}\Big(R-2\Lambda+
      \tilde{\mathcal L}[u',v',w]\Big).}
      \label{eq:ehactionIZ}
    \end{equation}

    We need to consider the variations of the functionals
    \begin{eqnarray*}\tilde
    S{}_1=\int{d^4x\sqrt{|g|}\big(R-2\Lambda\big)}&\mbox{and}&
    \tilde S{}_2=\int{d^4x\sqrt{|g|}\tilde{\mathcal
    L}[u',v',w]}.
    \end{eqnarray*}

     The variation of the first of these functionals is given by the equation
    (\ref{eq:varderivativeS1}). The variation of the second
    functional is

    \begin{equation}
      \delta \tilde
      S{}_2=\int{d^4x\sqrt{|g|}\big\{\frac{\delta\tilde{\mathcal
      L}[u',v,w]}{\delta g^{\underline{\alpha\beta}}}-\frac12g_{\underline{\alpha\beta}}
      \tilde{\mathcal
      L}[u',v',w]\big\}\delta
      g^{\underline{\alpha\beta}}}.
      \label{eq:varderivativeS2Z}
    \end{equation}

    With respect to the Quotient Rule, the variational
    derivatives $\delta\tilde{\mathcal L}[u',v',w]/\delta
    g^{\underline{ij}}$ are the components of the tensor
    $\hat{\mathcal{\tilde{V}}}$ of the type $(0,2)$, i.e.

    \begin{equation}
      \frac{\delta\tilde{\mathcal L}[u',v',w]}{\delta
    g^{\underline{ij}}}=\tilde{\mathcal{V}}{}_{ij}.
    \label{eq:ZVij}
    \end{equation}

    If sum the equations (\ref{eq:varderivativeS1}) and
    (\ref{eq:varderivativeS2Z}), we will obtain

    \begin{equation}
      \delta\tilde{S}=\int{d^4x\sqrt{|g|}\big\{R_{\alpha\beta}-
      \frac12Rg_{\underline{\alpha\beta}}+\Lambda g_{\underline{\alpha\beta}}
      +\tilde{\mathcal{V}}{}_{\alpha\beta}-
      \frac12g_{\underline{\alpha\beta}}\tilde{\mathcal L}[u',v',w]\big\}
      \delta g^{\underline{\alpha\beta}}}.
      \label{eq:variationSZ}
    \end{equation}

    The right side of the last equation vanishes if and only if

    \begin{equation}
      R_{ij}-\frac12Rg_{\underline{ij}}+\Lambda
      g_{\underline{ij}}=-\tilde{\mathcal V}{}_{ij}+
      \frac12g_{\underline{ij}}
      \tilde{\mathcal
      L}[u',v',w],
      \label{eq:einsteinequationsofmotionZI}
    \end{equation}

    \noindent which are the corresponding Einstein's equations of
    motion.

    It holds the following theorem.

    \begin{thm}
      With respect to the family of the Einstein's equations of motion
      \emph{(\ref{eq:einsteinequationsofmotionZI})}, the family of
      energy-momentum tensors and the family of their traces are

      \begin{eqnarray}
        \tilde T_{ij}=-\tilde{\mathcal V}{}_{ij}+
      \frac12g_{\underline{ij}}
      \tilde{\mathcal
      L}[u',v',w]&\mbox{and}&
        \tilde T^\alpha_\alpha[u',v',w]=
        -\tilde{\mathcal V}{}^\alpha_\alpha+2
        \tilde{\mathcal L}[u',v',w].
        \label{eq:tracesetIZ}
      \end{eqnarray}

      The pressure, energy-density and state-parameter
       may be obtained by the substituting the
       equation \emph{(\ref{eq:tracesetIZ})}
       into the equations \emph{(\ref{eq:omega}, \ref{eq:omegacrf})}.

       The $p$EQM and $\rho$EQM are

       \begin{align}
         &\dfrac13R_{\alpha\beta}u^\alpha u^\beta+\dfrac16R-\Lambda=
         \dfrac13\tilde{\mathcal V}{}^\alpha_\alpha-
         \dfrac13\tilde{\mathcal V}{}_{\alpha\beta}u^\alpha u^\beta-
         \dfrac12\tilde{\mathcal L}[u',v',w],
         \\
         &R_{\alpha\beta}u^\alpha u^\beta-\dfrac12R+\Lambda=-
         \tilde{\mathcal V}{}_{\alpha\beta}u^\alpha u^\beta+
         \dfrac12\tilde{\mathcal L}[u',v',w],
       \end{align}

       \noindent in the reference system $u^i$ and

       \begin{align}
         &\dfrac13R_{00}+\dfrac16R-\Lambda=
         \dfrac13\tilde{\mathcal V}{}^\alpha_\alpha-
         \dfrac13\tilde{\mathcal V}{}_{00}-\dfrac12\tilde{\mathcal
         L}[u',v',w],
         \\
         &R_{00}-\dfrac12R+\Lambda=
         -\tilde{\mathcal
         V}{}_{00}+\dfrac12g_{\underline{00}}\tilde{\mathcal
         L}[u',v',w],
       \end{align}

       \noindent in the comoving reference system.\qed
    \end{thm}

    The part $\tilde{\mathcal L}[u',v',w]$
    which corresponds to the matter with respect to the space
    $\mathbb{G\tilde R}{}_4$ and the corresponding part $\mathcal
    L_M$ which corresponds to the matter with respect to the
    space $\mathbb{GR}_4$ satisfy the equality
    \begin{equation*}
    \tilde{\mathcal
    L}[u',v',w]=
    \mathcal L_M+\big(\tilde{\mathcal
    L}[u',v',w]-\mathcal L_M\big).
    \end{equation*}

    Indirectly, we will find the difference between the
    energy-momentum tensors, pressures, energy-densities and state
    parameters obtained with respect to the families (\ref{eq:LMN})
    and (\ref{eq:LMIZ}) in the following section.

    \section{Linearity}

    The question which arises is how much would we change the
    energy-momentum tensor, the pressure, the energy-density and the
    state-parameter obtained with respect to the part $\mathcal L_M$
    if we get the part $l\mathcal L_M+f\mathcal F$ for some field $\mathcal F$ and
    the real or complex scalars $l$ and $f$. We will answer this
    question more generally below.

    Let us consider the field

    \begin{equation}
      \overset\diamond{\mathcal L}{}_{M}=\underset1\alpha \underset1{{\mathcal
      L}}{}_{M}+\ldots+\underset s\alpha\underset s{{\mathcal
      L}}{}_{M},
      \label{eq:couplingL}
    \end{equation}

    \noindent for some $s\in\mathbb N$, fields $\underset1{\mathcal L}$,
    \ldots, $\underset s{\mathcal L}$ and real or complex
    coefficients $\underset1\alpha,\ldots,\underset s\alpha$.

    The corresponding Einstein-Hilbert action is

    \begin{equation}
      \overset\diamond S=\int{d^4x\sqrt{|g|}\big(R-2\Lambda+\overset\diamond{\mathcal
      L}_M\big)}.
      \label{eq:LMcoupling}
    \end{equation}

    The last equation may be equivalently treated as the equation
    (\ref{eq:ehactionIZ}) after changing the field $\tilde{\mathcal
    L}{}_M[u',v',w]$ by the field $\overset\diamond{\mathcal
    L}{}_M$. Any of the fields $\underset r{\mathcal L}$,
    $r=1,\ldots,s$, generates the corresponding tensor $\underset
    r{\hat{\mathcal V}}$ analogous to the tensor $\hat{\mathcal V}$
    whose components are given by the equation (\ref{eq:ZVij}).
    For this reason, the components of the tensor
    $\overset{\hat{{\diamond}}}{\mathcal V}$ obtained with respect to the field
    $\overset\diamond{\mathcal L}{}_M$ are

    \begin{eqnarray}
      \overset\diamond{\mathcal V}{}_{ij}=
      \underset1\alpha\underset1{\mathcal V}{}_{ij}+\ldots+
      \underset s\alpha\underset s{\mathcal
      V}{}_{ij}.
      \label{eq:l=l1+...+lstrace}
    \end{eqnarray}

    Hence, the Einstein's equations of motion are

    \begin{equation}
      \aligned
      R_{ij}-\frac12Rg_{\underline{ij}}+2\Lambda g_{\underline{ij}}&=-\sum_{r=1}^s
      {\underset r\alpha\underset r{{\mathcal
      V}}{}_{ij}}+\frac12g_{\underline{ij}}\sum_{r=1}^s
      {\underset r\alpha\underset r{{\mathcal
      L}}}.
      \endaligned\label{eq:l=l1+...+lsmotionin}
    \end{equation}

    It is the set of the corresponding Einstein's equations of motion.

    The equations of motion (\ref{eq:l=l1+...+lsmotionin})
    may be rewritten as

    \begin{equation}
      R_{ij}-\frac12Rg_{\underline{ij}}+2\Lambda g_{\underline{ij}}=
      -\sum_{r=1}^s{\underset r\alpha\big(
      \underset r{{\mathcal V}}{}_{ij}-
      \frac12g_{\underline{ij}}\underset r{{\mathcal L}}
      \big)}.\label{eq:l=l1+...+lsmotion}
    \end{equation}

    Based on the equations ( \ref{eq:omega}, \ref{eq:omegacrf}, \ref{eq:l=l1+...+lsmotion}),
    we obtain the following
    equalities

    \begin{table}[h]
    \centering
      \begin{tabular}{|l:l|}
      \hline
      \mbox{In reference system }$u^i$&
      \mbox{In comoving reference system $u^i=\delta^i_0$}\\
      \hdashline
      \multicolumn{2}{|c|}{$\overset\diamond T{}_{ij}=-\sum_{r=1}^s{
      \underset r\alpha\big({\underset r{\mathcal V}}{}_{ij}-
      \frac12g_{\underline{ij}}{\underset r{\mathcal L}}
      \big)}$}\\
      \multicolumn{2}{|c|}{$\overset\diamond{T}{}^\alpha_\alpha=
      -\sum_{r=1}^s{\underset r\alpha\Big(\underset r{\mathcal V}{}^\alpha_\alpha-
      2\underset r{{\mathcal
      L}}\Big)
      }$}\\\hdashline
        $\overset\diamond p=-\dfrac13\sum_{r=1}^s{\underset r\alpha
        \big(\underset r{\mathcal V}{}_{\alpha\beta}u^\alpha u^\beta-
        \underset r{\mathcal V}{}^\alpha_\alpha-\underset r{\mathcal L}\big)}
      $&$\overset\diamond p=-\dfrac13\sum_{r=1}^s{\underset r\alpha
        \big(\underset r{\mathcal V}{}_{00}-
        \underset r{\mathcal V}{}^\alpha_\alpha-\underset r{\mathcal L}\big)}
      $\\
      $\overset\diamond\rho=-
      \sum_{r=1}^s{\underset r\alpha\Big(
      \underset r{{\mathcal V}}{}_{\alpha\beta}
      u^\alpha u^\beta-\frac12\underset r{\mathcal L}
      \Big)}$&$
      \overset\diamond\rho{}_0=-
      \sum_{r=1}^s{\underset r\alpha
      \big(\underset r{{\mathcal V}}{}_{00}-
      \frac12\underset r{\mathcal
      L}\big)}$\\
      $\overset{\diamond}{\omega}=\overset\diamond
      p\overset\diamond\rho{}^{-1}$&$
      \overset\diamond\omega{}_0=\overset\diamond
      p{}_0\overset\diamond\rho{}_0^{-1}$\\\hline
      \end{tabular}
      \caption{Linear combinations of energy-momentums, pressures and energy-densities}
      \label{tab:lincombsepes}
    \end{table}

    With respect to the equalities (\ref{eq:l=l1+...+lstrace}) and
    the expressions in
    the Table \ref{tab:lincombsepes},
    we proved the validity of the following theorem.

    \begin{thm}
      The energy-momentum tensor and its trace, the pressure and the
      energy-density are linear by the summands into the
      lagrangian which corresponds the matter. Their values are
      equal  to the linear combinations
      of the corresponding values generated by the separate
      summands of the lagrangian.

      The corresponding $p$EQM and $\rho$EQM of the system
      are the linear combinations
      of the $p$EQMs and $\rho$EQMs of the separate subsystems
      generated by the summands in the lagrangian which
      corresponds to the whole system.

      The state-parameter is not linear as the previous
      magnitudes.\qed
    \end{thm}

    \section{Conclusion}

    In the section 2, we geometrically interpreted the Madsen's
    formulae about energy-momentum tensors, pressures and energy-densities.
    With respect to these interpretations, we computed these
    magnitudes \big(see the Table \ref{tab:pressdenstp}\big).

    In the section 3, we expressed the energy-momentum tensor with
    respect to the curvature tensors of
    a generalized Riemannian space in the sense of Eisenhart's definition.  We also
    proved that the part $\mathcal L_M$ generates two generalized
    Riemannian spaces in the sense of Eisenhart's definition. In
    this section, it is concluded that the anti-symmetric part of a
    metric tensor corresponds to a matter.

    In the section 4, we analyzed the differences and similarities
    between the results presented in the Shapiro's paper \cite{shapiro}
    and the model of the generalized Riemannian space involved by S.
    Ivanov and M. Lj. Zlatanovi\'c \cite{zlativanov}. We explicitly obtained
    the corresponding energy-momentum tensor but the pressure,
    energy-density and state parameter may be obtained with respect
    to the corresponding formulae from the section 2.

    In the Sections 3 and 4, we obtained the systems of equations
    for the equilibriums between symmetric affine connections and
    torsions \big(the systems $p$EQM and $\rho$EQM\big).

    In the section 5, we analyzed the linearity of the
    energy-momentum
    tensors, energy-densities, pressures and state parameters under
    summing of matter fields \big(see the Table \ref{tab:lincombsepes}\big).
    We also analyzed the linearity of
    $p$EQM and $\rho$EQM in this section.

    \section*{Acknowledgements}

    This paper is financially supported by Serbian Ministry of
    Education, Science and Technological Development, Grant No.
    174012.


\end{document}